\documentclass[reqno]{amsart}

\usepackage{amsmath}
\usepackage{amssymb}
\usepackage{mathrsfs} 
\usepackage{amscd}
\usepackage{blkarray}
\usepackage{times}
\usepackage[all]{xy}

\def\dim{\operatorname{dim}}
\def\ord{\operatorname{ord}}
\def\Ker{\operatorname{Ker}}

\def\CH{\operatorname{CH}}
\def\Hom{\operatorname{Hom}}

\def\Ext{\operatorname{Ext}}
\def\cl{\operatorname{cl}}

\def\hom{\mathrm{hom}}
\def\alg{\mathrm{alg}}
\def\sgn{\mathrm{sgn}}
\def\Tr{\mathrm{Tr}}
\def\holo{\mathrm{holo}}

\def\C{\mathbb{C}}\def\Q{\mathbb{Q}}\def\R{\mathbb{R}}\def\Z{\mathbb{Z}}
\def\fa{\mathfrak{a}}

\def\ol#1{\overline{#1}}\def\wt#1{\widetilde{#1}}
\def\os#1#2{\overset{#1}{#2}}\def\us#1#2{\underset{#1}{#2}}
\def\ot{\otimes}
\def\ra{\rightarrow}
\def\lra{\longrightarrow}

\def\a{\alpha}\def\b{\beta}\def\g{\gamma}\def\d{\delta}\def\pd{\partial}
\def\io{\iota}\def\k{\kappa}\def\z{\zeta}\def\vphi{\varphi}
\def\o{\omega}
\def\s{\sigma}
\def\vD{\varDelta}
\def\vG{\varGamma}
\def\sO{\mathscr{O}}
\def\bI{\mathbf{I}}

\def\dlog#1{\frac{d#1}{#1}}
\def\angle#1{\langle #1 \rangle}
\def\inner#1#2{\langle#1,#2\rangle}

\def\F#1#2#3#4#5{\,_3F_2\left({{#1,#2,#3} \atop {#4,#5}} ;1\right)}
\def\Fx#1#2#3#4#5{\,_3F_2\left({{#1,#2,#3} \atop {#4,#5}} ;x\right)}
\def\G#1#2{\varGamma\left({{#1} \atop {#2}}\right)}

\theoremstyle{plain}
    \newtheorem{theorem}{Theorem}[section]
    \newtheorem{proposition}[theorem]{Proposition}
    \newtheorem{lemma}[theorem]{Lemma}
    \newtheorem{corollary}[theorem]{Corollary}
\theoremstyle{definition}
    
    \newtheorem{remark}[theorem]{Remark}
    \newtheorem{example}[theorem]{Example}
    \newtheorem{assumption}[theorem]{Assumption}
\numberwithin{equation}{section}

\begin{document}
\title{On the Abel-Jacobi maps of Fermat Jacobians}
\author{Noriyuki Otsubo}
\address{Department of Mathematics and Informatics, Chiba University, Yayoicho 1-33, Inage, Chiba, 263-8522 Japan} 
\email{otsubo@math.s.chiba-u.ac.jp}
\date{March 1, 2010.}
\begin{abstract}
We study the Abel-Jacobi image of the Ceresa cycle $W_k-W_k^-$, where $W_k$ is the image of the $k$-th symmetric product of a curve $X$ on its Jacobian variety. For the Fermat curve of degree $N$, we express it in terms of special values of generalized hypergeometric functions and give a criterion for the non-vanishing of $W_k-W_k^-$ modulo algebraic equivalence, which is verified numerically for some $N$ and $k$.  
\end{abstract}
\maketitle
%
%%%%%%%%%%%%%%%%%%%%%%%%%%%%%%%%
%%%%%%%%%%%%%%%%%%%%%%%%%%%%%%%%
\section{Introduction}

Let $A$ be a smooth projective variety of dimension $g$ over $\C$. 
The Chow group $\CH_k(A)$ of algebraic cycles of dimension $k$ on $A$  modulo rational equivalence has the subgroups
$$\CH_k(A) \supset \CH_k(A)_\hom \supset \CH_k(A)_\alg$$
of homologically and algebraically trivial cycles, respectively. 
Consider the Abel-Jacobi map
$$\Phi_k \colon \CH_k(A)_\hom \ra J_k(A)$$
to the Griffiths intermediate Jacobian, which is a complex torus associated to the Hodge structure 
$H_{2k+1}(A)(-k)$. 
Since the the image of an algebraically trivial cycle has a certain Hodge-theoretic property, 
the Abel-Jacobi map gives us a criterion for a homologically trivial cycle not to be 
algebraically trivial, in other words, non-trivial in the Griffiths group $\CH_k(A)_\hom/\CH_k(A)_\alg$. 

Let $X$ be a smooth projective curve of genus $g \geq 3$ and $A$ be its Jacobian variety. 
By choosing a base point, $X$ is embedded into $A$ and the $k$-th symmetric product of $X$ 
defines an algebraic cycle of dimension $k$ on $A$, which is historically denoted by $W_k$. 
Since the inversion of $A$ acts trivially on the homology groups of even degree, 
the cycle $W_k-W_k^-$ is homologically trivial. This is called the Ceresa cycle and 
Ceresa  \cite{ceresa} showed that
$W_k-W_k^-$ is not algebraically trivial for a generic curve $X$ and $1\leq k \leq g-2$. 
Note that $\Phi_{g-1}$ is an isomorphism by the Abel-Jacobi theorem. 

Harris \cite{harris-1} defined an invariant called the pointed harmonic volume using Chen's iterated integrals, and showed that it calculates $\Phi_1(W_1-W_1^-)$. 
It is not easy, however, to compute the harmonic volume of a given curve. 
Harris \cite{harris-2} calculated numerically the harmonic volume of the Fermat quartic curve  
and showed that $W_1-W_1^-$ is not algebraically trivial. 
In this case, Bloch \cite{bloch} showed by an $\ell$-adic method that $W_1-W_1^-$ is even non-torsion modulo algebraical equivalence. 
For $k>1$, Faucette introduced the higher harmonic volume \cite{faucette-2} and used it to prove that $W_2-W_2^-$ is not algebraically trivial for a certain finite covering of the Fermat quartic curve \cite{faucette-1}. 

Rather recently, Tadokoro proved similar results as Harris' for the Klein quartic curve \cite{tadokoro}, which is a quotient of the Fermat septic curve, and for the Fermat sextic curve \cite{tadokoro-2}. 
Moreover, he showed that the iterated integrals on the Fermat curve are expressed by special values of Barnes'  generalized hypergeometric function ${}_3F_2$. 

In this paper, we generalize the above results of Harris and Tadokoro and calculate the harmonic volume of the Fermat curve of arbitrary degree. 
Moreover, we reduce the study of $\Phi_k(W_k-W_k^-)$ for any $k$ to the case $k=1$, and hence to the calculation of the harmonic volume.  
As a result, we shall prove: 

\begin{theorem}\label{main-theorem} 
Let $X$ be the Fermat curve of degree $N \geq 4$, and $k$ be an integer with 
$1 \leq k \leq g-2$ where $g=(N-1)(N-2)/2$. 
\begin{enumerate}
\item 
If 
$$ 
\frac{\vG\bigl(1-\tfrac{h}{N}\bigr)^4}{\vG\bigl(1-\tfrac{2h}{N}\bigr)^2} \  
\F{\tfrac{h}{N}}{\tfrac{h}{N}}{1-\tfrac{2h}{N}}{1}{1}$$
is not an element of $\Q(\mu_N)$
for some integer $h$ with $0<h<N/2$, $(h,N)=1$, then  
$W_k-W_k^-$ is non-torsion modulo algebraic equivalence.  
\item 
If 
$$k!  \cdot 2N^{2k} \sum_{0<h<N/2 \atop   (h,N)=1} \frac{\vG\bigl(1-\tfrac{h}{N}\bigr)^4}{\vG\bigl(1-\tfrac{2h}{N}\bigr)^2} \ 
\F{\tfrac{h}{N}}{\tfrac{h}{N}}{1-\tfrac{2h}{N}}{1}{1}$$
is not a rational integer, then 
$k! (W_k-W_k^-)$ is non-trivial modulo algebraic equivalence. 
This is indeed the case for $N \leq 1000$, $k=1$, or $N \leq 8$ and any $k$. 
\end{enumerate}
\end{theorem}

Our key ingredients are as follows. 
First, we extend the coefficients to $\sO=\Z[\mu_N]$, so that the cohomology $H^1(X,\Z) \ot \sO$ has 
$2g$ linearly independent elements $\{\vphi^{a,b}\}$ which are eigenvectors with respect to a group action.  
The harmonic volume is a functional on a subgroup $K_\Z \ot H^1_\Z \subset \ot^3 H^1(X,\Z)$ with values in $\R/\Z$. 
We evaluate it at $\vphi^{a_1,b_1} \ot \vphi^{a_2,b_2} \ot \vphi^{a_2,b_3}$ satisfying certain assumption. 
Then, it takes the value in $(\sO \ot \R)/\sO$, which is isomorphic to the direct sum of copies of $\C$ for the  infinite places of $\Q(\mu_N)$. 
The harmonic volume is defined by iterated integrals of length $\leq 2$, and the iterated integral of $\vphi^{a,b}$ coincides with an integral representation of the value at $1$ of Barnes' generalized hypergeometric function ${}_3F_2$. 
This value is also written as a special value of Appell's generalized hypergeometric function $F_3$ of two variables. 

The intermediate Jacobian $J_k(A)$ is regarded as a functional on $\wedge^{2k+1}H^1(X,\Z)$ with values in $\R/\Z$. The Abel-Jacobi image $\Phi_k(W_k-W_k^-)$, after multiplied by $k!$,  
is written as the Pontryagin product of $\Phi_1(W_1-W_1^-)$ and a cycle class, so we are reduced to the case $k=1$. 
We shall take up a special element $\vphi^{a_1,b_1} \wedge \cdots \wedge \vphi^{a_{2k+1},b_{2k+1}}$ to 
evaluate $k! \cdot \Phi_k(W_k-W_k^-)$. It takes the value in $(\sO \ot \R)/\sO$ and we obtain 
Theorem \ref{main-theorem} (i). By taking the trace to $\R/\Z$, we obtain (ii), which can be computed  numerically. Since the above element has the Hodge type $(k+2,k-1) + (k-1,k+2)$, the non-vanishing implies that $W_k-W_k^-$ is not algebraically trivial. 

Finally, let us mention a connection to number theory. 
If a curve $X$ is defined over a number field, then so is its Jacobian $A$ and $\CH_k(A)$ is conjectured to be a finitely generated abelian group. 
The conjecture of Swinnerton-Dyer \cite{s-d} (generalizing the Birch-Swinnerton-Dyer conjecture) states  that 
$$\mathrm{rank}\CH_k(A)_\hom = \ord_{s=k+1} L(h^{2k+1}(A),s). $$
Here, $L(h^{i}(A),s)$ is the $L$-function associated to the $i$-th cohomological motive of $A$, 
which conjecturally is continued analytically to the whole complex plane and satisfies a functional equation with respect to $s \leftrightarrow 2k+2-s $. 
If $X$ is a Fermat curve, then by Weil \cite{weil-jacobi}, 
$L(h^{2k+1}(A),s)=L(\wedge^{2k+1} h^1(X),s)$ is the product of the $L$-functions 
of Jacobi-sum Hecke characters of cyclotomic fields. 
Therefore, the non-torsionness of $\Phi_k(W_k-W_k^-)$ at an element of $\wedge^{2k+1} H^1(X,\Z) \ot \sO$ 
should be related with the vanishing of the corresponding $L$-function at $s=k+1$, and our result suggests a connection between hypergeometric values and the Jacobi-sum Hecke $L$-functions. 
Such a connection is also found in \cite{otsubo}, where the Beilinson regulator is written in terms of similar hypergeometric values.  

This paper is organized as follows. In \S 2, we recall the Abel-Jacobi map, the pointed harmonic volume,  and Harris' theorem as reworked by Pulte. In \S 3, after introducing the Pontryagin product, we reduce the calculation of the Able-Jacobi image of the Ceresa cycle to the case $k=1$. In \S 4, we calculate the iterated integrals and the harmonic volume of the Fermat curve, and express them as special values of generalized hypergeometric functions. 
Finally, we evaluate the Abel-Jacobi image at a special element and prove Theorem \ref{main-theorem}. 

\medskip
The author would like to thank Tomohide Terasoma for his suggestion to read the works of Harris and Tadokoro. I am also grateful to Seidai Yasuda and Yuuki Tadokoro for their valuable comments.

%%%%%%%%%%%%%%%%%%%%%%%%%%%%%%%%
\section{Abel-Jacobi maps and harmonic volume}

In this section, we first recall the necessary materials on the Griffiths intermediate Jacobians and the Abel-Jacobi maps. 
Then, we introduce iterated integrals and the extension of mixed Hodge structures associated to the fundamental group, which is the most natural way due to Pulte to introduce Harris' pointed harmonic volume. 

%%%%%%%%%%%%%%%%%%%
\subsection{Abel-Jacobi maps}

Recall \cite{hodge-2} that a {\em Hodge structure} of weight $w$ is a $\Z$-module $H_\Z$ of finite rank equipped with a finite descending filtration $F^\bullet H_\C$ on $H_\C := H_\Z \ot_\Z \C$ such that $H_\C \simeq  \bigoplus_{p+q=w} H^{p,q}$ with $H^{p,q} = F^pH_\C \cap \ol{F^qH_\C}$. 
If $X$ is a smooth projective variety over $\C$, the cohomology group $H^n(X,\Z)$ 
underlies a Hodge structure of weight $n$, which we denote by $H^n(X)$.  
Let $\Z$ denote the Hodge structure of weight $0$ with $H_\Z=\Z$, $H_\C=H^{0,0}$. 
For a Hodge structure $H$ of weight $w$, we denote by $H(r)$ 
the Hodge structure of weight $w-2r$ with the {\em same} $\Z$-lattice $H_\Z$ and the shifted filtration 
$F^pH(r)_\C=F^{p+r}H_\C$. 

If $H$ is a Hodge structure of weight $-1$, then 
$$JH := H_\C/(F^0H_\C + H_\Z)$$
defines a complex torus, and this construction is functorial in $H$. 
In particular, for a smooth projective variety $X$ over $\C$, $H_{2k+1}(X)(-k)$ is a Hodge structure of weight $-1$, 
and the associated complex torus 
$$J_k(X) := JH_{2k+1}(X)(-k)$$
is the $k$-th {\em intermediate Jacobian} of Griffiths. 
Since 
$$H_{2k+1}(X,\C)/F^{-k}H_{2k+1}(X,\C) \simeq (F^{k+1}H^{2k+1}(X,\C))^*$$
where ${}^*$ denotes the $\C$-linear dual, we have an identification
\begin{equation}\label{int-jac-complex}
J_k(X) \simeq (F^{k+1}H^{2k+1}(X,\C))^* / H_{2k+1}(X,\Z). 
\end{equation}
By the Poincar\'e duality, we have also
$$J_k(X) \simeq H^{2d-2k-1}(X,\C)/(F^{d-k}H^{2d-2k-1}(X,\C)+H^{2d-2k-1}(X,\Z))$$
which is the original definition of $J^{d-k}(X)$ by Griffiths \cite{griffiths}. 

For a Hodge structure of weight $-1$, put
$$J_\R H = H_\R / H_\Z. $$
Then the inclusion $H_\R \ra H_\C$ induces an isomorphism 
$$J_\R H \os{\sim}{\ra} JH$$
of real Lie groups. Its inverse is induced by 
$$H_\C = F^0H_\C \oplus \ol{F^0 H_\C} \ra \ol{F^0H_\C} \ra H_\R$$
where the first map is the projection and the second map sends $\a$ to $\a + \ol\a$. 
Therefore, for a Hodge structure $H$ of weight $1$, we have an isomorphism of real tori
$$\Hom(H_\Z, \R/\Z) \simeq J\Hom(H,\Z). $$
In particular, we have
\begin{equation}\label{int-jac-real}
J_k(X) \simeq \Hom(H^{2k+1}(X,\Z), \R/\Z).
\end{equation}

Let $Z$ be an algebraic cycle on $X$ of dimension $k$. Then the cycle class map
$$\cl_k \colon \CH_k(X) \ra H_{2k}(X,\Z)$$
defines an element of 
$$\Hom(\Z,H_{2k}(-k)) = H_{2k}(X,\Z) \cap F^{-k}H_{2k}(X,\C)$$
where $\Hom$ means the homomorphisms of Hodge structures.   
Let $\CH_k(X)_\hom$ denote the kernel of $\cl_k$. 
If $\cl_k(Z)=0$, then there exists a topological $(2k+1)$-chain $W$ with $\pd W=Z$. 
The functional $\eta \mapsto \int_W \eta$ on $F^{k+1}H^{2k+1}(X,\C)$ is well-defined up to periods. 
By the identification \eqref{int-jac-complex},  the {\em Abel-Jacobi map} of dimension $k$ 
$$\Phi_k \colon \CH_k(X)_\hom \ra J_k(X)$$ 
is defined. 
If $k=\dim(X)-1$, $J_k(X)$ is nothing but the Picard variety of $X$ and $\Phi_k$ is an isomorphism by the classical Abel-Jacobi theorem. 

Let $\CH_k(X)_\alg \subset \CH_k(X)_\hom$ be the subgroup of algebraically trivial cycles, i.e. those cycles obtained as $p_*(\vG \cdot q^*D)$ where $D$ is a divisor of degree $0$ on a curve $C$, 
$\vG$ is an algebraic $(k+1)$-cycle on $X \times C$, and $p\colon X \times C \ra X$, $q\colon X \times C \ra C$ are the projections. 
One sees that if $Z \in \CH_k(X)_\alg$, then $\Phi_k(Z)$ is contained in 
$$F^{-k-1}H_{2k+1}(X,\C) / (F^{-k}H_{2k+1}(X,\C) + H_{2k+1}(X,\Z)).$$ 
Therefore, under the identifications \eqref{int-jac-complex} and  \eqref{int-jac-real}, 
the image $\Phi_k(Z)$ as a functional vanishes on 
$$F^{k+2}H^{2k+1}(X,\C), \quad H^{2k+1}(X,\Z) \cap (F^{k+2}+\ol{F^{k+2}})$$
respectively.  

Recall \cite{hodge-2} that a {\em mixed Hodge structure} is  a $\Z$-module $H_\Z$ of finite rank 
equipped with an increasing weight filtration $W_\bullet H_\Q$ 
and a decreasing Hodge filtration $F^\bullet H_\C$ such that for each $w$, $\mathrm{Gr}_w^W H_\Q$ with the induced filtration $F^\bullet$ is a $\Q$-Hodge structure of weight $w$. 
For mixed Hodge structures $A$, $B$, let 
$\Ext(A,B)$ denote the group of congruence classes of extensions of mixed Hodge structures (see \cite{carlson}),  
i.e. exact sequences 
$$0 \ra B \ra E \ra A \ra 0$$
of mixed Hodge structures with the natural equivalence relation and the Baer sum. 

If $A$, $B$ are Hodge structures of weights $w$, $v$, respectively, 
then $\Hom(A,B)$ is a Hodge structure of weight $v-w$. 
By Carlson \cite{carlson}, when $v-w=-1$,  
we have an isomorphism 
$$\Ext(A,B) \simeq J\Hom(A,B)$$
given as follows. 
For an extension $E$ as above, there is   
a section $s_F \colon A_\C \ra E_\C$ compatible with the Hodge filtrations 
and a retraction 
$r_\Z \colon E_\Z \ra B_\Z$ of $\Z$-modules.  
We associate to $E$ the class of $r_\Z\circ s_F$, which is well-defined. 
If we apply this to 
$$H_{2k+1}(X)(-k) = \Hom(H^{2k+1}(X)(k),\Z),$$ 
we obtain another identification 
\begin{equation*}
J_k(X)  \simeq \Ext(H^{2k+1}(X)(k),\Z).
\end{equation*}

%%%%%%%%%%%%%%%%%%%%%%%%%%%%%%%%
\subsection{Harmonic volume}
Let $X$ be a smooth projective variety over $\C$. 
Recall that, for smooth $1$-forms $\vphi_i$ ($i=1,2$) and a piecewise smooth path $\g$ on $X$, 
Chen's iterated integral (of length $2$) is defined by 
$$\int_\g \vphi_1\vphi_2 = \int_0^1 \left(\int_0^{t_2} f_1(t_1) dt_1 \right) f_2(t_2) dt_2$$
where $\g^*\vphi_i=f_i(t) dt$.  
For paths $\g$, $\g'$ with $\g(1)=\g'(0)$, let $\g\cdot \g'$ denote the composition. Then, we have a formula 
\begin{equation}\label{formula-1}
\int_{\g \cdot \g'} \vphi \vphi' = \int_\g \vphi \vphi' + \int_\g\vphi \int_{\g'} \vphi' + \int_{\g'} \vphi \vphi'.  
\end{equation}
It follows that 
\begin{equation}\label{formula-2}
\int_\g \vphi \vphi' + \int_{\g^{-1}} \vphi \vphi'= \int_\g\vphi \int_{\g} \vphi' ,
\end{equation}
and
\begin{equation}\label{formula-3}
\int_{\a^{-1}\cdot \g \cdot \a}\vphi\vphi' = \int_{\g}\vphi\vphi' - \int_\a \vphi\int_\g\vphi' +\int_\g \vphi \int_\a \vphi' 
\end{equation}
for a loop $\g$ and a path $\a$ with $\g(0)=\a(0)$. 
An iterated integral is {\em closed} if its value at $\g$ depends only on its homotopy class relative to the endpoints. 

Fix a base point $x \in X$ and consider the fundamental group $\pi_1(X,x)$. 
Let $\fa$ be the augmentation ideal of the group ring $\Z\pi_1(X,x)$, i.e. 
the kernel of the degree map
$$\Z\pi_1(X,x) \ra \Z; \quad \sum n_i \g_i \mapsto \sum n_i. $$ 
Then, by Chen's $\pi_1$ de Rham theorem  (\cite{chen-trans}, see also \cite{hain-1}),  
$\Hom(\Z\pi(X,x)/\fa^{s+1}, \R)$ is generated by closed iterated integrals of length $\leq s$. 
Using this, Hain \cite{hain-1} defined a mixed Hodge structure on 
$\Z\pi_1(X,x)/\fa^{s}$ such that the natural map $\Z\pi_1(X,x)/\fa^{s} \ra \Z\pi_1(X,x)/\fa^{t}$ for $s \geq t$ is a morphism of mixed Hodge structures (there is a different approach by Morgan \cite{morgan}). 

Consider the exact sequence of mixed Hodge structures
\begin{equation}\label{extension}
0 \ra \fa^2/\fa^3 \ra  \fa/\fa^3 \ra \fa/\fa^2 \ra 0.
\end{equation}
The map $\pi_1(X,x) \ra \fa/\fa^2; \g \mapsto \g-1$ is well-defined 
and induces an isomorphism
$$H_1(X,\Z) \os{\simeq}{\lra} \fa/\fa^2$$
of Hodge structures of weight $-1$. 
On the other hand, if we put
$$K= \Ker (\cup: \ot^2 H^1(X) \ra H^2(X)),$$  
then the multiplication
$\fa/\fa^2 \ot \fa/\fa^2 \ra \fa^2/\fa^3$
induces an isomorphism
$$\Hom(\fa^2/\fa^3,\Z) \os{\simeq}{\lra} K$$
of Hodge structures of weight $2$.  
Taking the dual of \eqref{extension}, we obtain an exact sequence 
\begin{equation*}
0 \ra H^1(X,\Z) \ra \Hom(\fa/\fa^3,\Z)  \ra K_\Z \ra 0
\end{equation*}
of mixed Hodge structures, so it defines an element 
$$m(x) \in J\Hom(K,H^1(X)).$$

Now, let $X$ be a projective smooth curve  over $\C$ of genus $g>0$. 
Let 
$$(\ , \ ) \colon H_1(X) \ot H^1(X) \ra \Z$$
be the canonical pairing and 
$$\langle \ , \ \rangle \colon H^1(X) \ot H^1(X) \ra \Z(-1)$$
be the cup product 
$\vphi \ot \vphi' \mapsto \int_X \vphi\wedge\vphi'$. 
By the Poincar\'e duality 
$$H_1(X)=H^1(X)^* \simeq H^1(X)(1),$$ 
we can regard $m(x)$ as an element of 
$$J\Hom(K \ot H^1(X)(1),\Z) \simeq \Hom(K_\Z \ot H^1(X)_\Z,\R/\Z).$$ 
This coincides with the {\em pointed harmonic volume} of Harris \cite{harris-1} (see also \cite{harris-book}): 
\begin{theorem}[Pulte \cite{pulte}, Theorem 3.9]\label{pulte}
As a functional, $m(x)$ sends 
$$\sum_i\left(\vphi^{(i)}_1 \ot \vphi^{(i)}_2\right) \ot \vphi_3 \in K_\Z \ot H^1(X)_\Z$$ 
to
$$\int_{\g_3} \left(\sum_i \wt\vphi^{(i)}_1\wt\vphi^{(i)}_2 +\eta\right)$$
where, 
$\wt\vphi_j^{(i)}$ is a harmonic $1$-form representing $\vphi_j^{(i)}$, 
$\eta$ is a (unique) $1$-form orthogonal to all closed $1$-forms such that 
$$\sum_i \vphi^{(i)}_1\wedge \vphi^{(i)}_2 +d\eta=0,$$ 
and $\g_3$ is a loop based at $x$ whose homology class is the Poincar\'e dual of $\vphi_3$.  
\end{theorem}

Here, for the Poincar\'e duality, we follow the convention of Harris (loc. cit.) : 
$\g$ is the Poincar\'e dual of $\a$ if 
$$(\g,\vphi')=\inner{\vphi'}{\vphi}$$ 
for all $\vphi' \in H^1$. Later, we shall restrict ourselves to the case when both $\vphi_1$, $\vphi_2$  are holomorphic, so that $\eta$ in the theorem is trivial.    

Let $A$ be the Jacobian variety of $X$. 
Choosing a base point $x$, $X$ is embedded into $A$ by
$$\io_x \colon X \hookrightarrow A; \quad y \mapsto [y]-[x]. $$
It induces isomorphisms 
$$H_1(X) \os{\simeq}{\lra} H_1(A), \quad H^1(A) \os{\simeq}{\lra} H^1(X)$$
which do not depend on the choice of $x$. 
We identify these and sometimes denote them just by $H_1$, $H^1$. 
Recall that the cup product induces an isomorphism 
$$\wedge^n H^1(A) \os{\simeq}{\ra} H^n(A). $$ 
One sees easily, for example using a symplectic basis of $H^1(X,\Z)$,  
that the wedge product $K \ot H^1 \ra \wedge^3 H^1$ is surjective, so we have an injection
\begin{equation}\label{injection}
J_1(A) \hookrightarrow J\Hom(K \ot H^1(1),\Z). 
\end{equation}

The image $W_1(x):=\io_x(X)$ defines an algebraic $1$-cycle on $A$. Since the inversion and a translation fix $H_2(A,\Z)$, we have $W_1(x)-W_1^\pm(y) \in \CH_1(A)_\hom$ for any $x, y \in X$. 
Then, Harris' result  \cite{harris-1} reworked by Pulte \cite{pulte}, Theorem 4.9, is: 

\begin{theorem}\label{harris-pulte}
Let $X$ be a smooth projective curve over $\C$ and $x$, $y \in X$. 
Under \eqref{injection}, 
$\Phi_1(W_1(x)- W_1^\pm (y))$
maps to  
$m(x)\mp m(y)$, respectively. 
\end{theorem}

%%%%%%%%%%%%%%%%%%%%%%%%%%%%%%%%%%%%%%%%%%%%%%%%
%%%%%%%%%%%%%%%%%%%%%%%%%%%%%%%%%%%%%%%%%%%%%%%%
\section{Reduction to $k=1$}

In this section, we recall product structures on the (co)homology groups 
of an Abelian variety and reduce the calculation 
of the Abel-Jacobi image of Ceresa cycles to the case $k=1$. 

\subsection{Products}
For an abelian variety $A$, let
$$\mu \colon A \times A \ra A, \quad \vD \colon A \ra A\times A$$
be the addition and the diagonal, respectively. 
The Pontryagin product 
$$* \colon \CH_p(A) \ot \CH_q(A) \ra \CH_{p+q}(A)$$
is defined to be the composition of the exterior product and the push-forward $\mu_*$. 
It makes $\CH_\bullet(A)$ a graded commutative ring. 
Similarly, we have the Pontryagin product on homology 
$$* \colon H_p(A) \ot H_q(A) \ra H_{p+q}(A),$$
which is a morphism of Hodge structures of weight $-p-q$. 
It is skew-symmetric, i.e. $\a * \b = (-1)^{pq} \b * \a$, and induces an isomorphism 
\begin{equation}\label{identification_homology}
\wedge^n H_1(A)\os{\simeq}{\lra} H_n(A). 
\end{equation}
If we identify these groups, the Pontryagin product is identified with the natural map
\begin{equation}\label{natural-map}
\wedge^pH_1(A) \ot \wedge^q H_1(A) \ra \wedge^{p+q}H_1(A).
\end{equation}
Pontryagin products on Chow groups and homology groups are compatible with the cycle maps
$$\cl_n \colon \CH_n(A) \ra H_{2n}(A,\Z).$$

On the other hand, 
the intersection product
$$\cdot \colon \CH^p(A) \ot \CH^q(A) \ra \CH^{p+q}(A)$$
and the cup product 
$$\cup \colon H^p(A) \ot H^q(A) \ra H^{p+q}(A)$$
are the composition of the exterior product and the pull-back $\vD^*$.
It induces an isomorphism 
$$\wedge^n H^1(A) \os{\sim}{\ra} H^n(A)$$
and under this identification, the cup product is identified with the natural map
$$\wedge^pH^1(A) \ot \wedge^q H^1(A) \ra \wedge^{p+q}H^1(A).$$

Therefore, the (co)homology rings $(H_\bullet(A),*)$, $(H^\bullet(A),\cup)$ are identified with the exterior algebras 
$\wedge^\bullet H_1(A)$, $\wedge^\bullet H^1(A)$, respectively, and are dual to each other. 
The canonical pairing 
$( \ , \  ) \colon H_n(A) \ot H^n(A) \ra \Z$
is identified with the pairing 
$$\wedge^n H_1(A) \ot \wedge^n H^1(A) \ra \Z$$
induced from $(\ , \ ) \colon H_1(A) \ot H^1(A) \ra \Z$, that is, 
$$(\a_1 \wedge \cdots \wedge \a_n, \vphi_1 \wedge \cdots \wedge \vphi_n) 
= \det((\a_i, \vphi_j)).$$ 

\begin{lemma}\label{dual-pont} 
Let 
$$\pi_{p,q} \colon \wedge^{p+q} H^1(A) \ra \wedge^p H^1(A) \ot \wedge^q H^1(A)$$
be the dual of the natural map \eqref{natural-map}. 
Then, 
for $\vphi = \vphi_1 \wedge \cdots \wedge \vphi_{p+q} \in \wedge^{p+q} H^1(A)$, 
we have
$$\pi_{p,q}(\vphi) = \sum_{\s \in S_{p,q}} \sgn(\s)  \vphi_{\s(1)}\wedge \cdots \wedge \vphi_{\s(p)}  
\ot \vphi_{\s(p+1)} \wedge \cdots \wedge \vphi_{\s(p+q)}$$
where we put $S_{p,q} =\{ \s \in S_{p+q}  \mid \s(1) < \cdots <\s(p), \s(p+1) < \cdots <\s(p+q)\}$. 
\end{lemma}

\begin{proof}
Fix a basis $\{\a_i \mid 1 \leq i \leq 2g\}$ of $H_1$ and let $\{\check\a_i\}$ be the dual basis of $H^1$, i.e. $(\a_i, \check\a_j)=\d_{i,j}$. 
For an injection $\tau \colon \{1,2,\dots,n\} \ra   \{1,\dots,2g\}$, we put 
$$\a_\tau = \a_{\tau(1)} \wedge \cdots \wedge \a_{\tau(n)}, \quad
\check\a_\tau = \check\a_{\tau(1)} \wedge \cdots \wedge \check\a_{\tau(n)}. $$
Then those $\a_\tau$ and $\check\a_\tau$ with order-preserving $\tau$ form bases of $\wedge^n H_1$ and $\wedge^n H^1$, respectively, dual to each other. 
Then, for order-preserving injections $\tau_p \colon \{1,2,\dots,p\} \ra \{1,\dots,2g\}$, $\tau_q\colon \{1,2,\dots,q\} \ra \{1,\dots,2g\}$ with $p+q=n$, 
$$
(\a_{\tau_p} \ot \a_{\tau_q},\pi_{p,q}(\check\a_\tau)) = (\a_{\tau_p} \wedge \a_{\tau_q}, \check\a_\tau)
$$
is non-trivial if and only if there exists $\s \in S_n$ such that
$$\tau_p(i) = \tau\s(i), \quad \tau_q(j)= \tau\s(p+j)$$
for all $1 \leq i \leq p$, $1 \leq j \leq q$,  
in which case, $\a_{\tau_p} \wedge \a_{\tau_q}=\sgn(\s) \a_\tau$, so the assertion for
 $\vphi =\check\a_\tau$ is proved. Since the right-hand side of the formula is multi-linear and skew-symmetric, the general case follows.  
\end{proof}

%%%%%%%%%%%%%%%%%%%%%%%%%
\subsection{Ceresa cycles of higher dimension}
Let $X^k$ denote the $k$-fold product of $X$ and $X_k$ denote the $k$-th symmetric product. 
Then we have a commutative diagram
$$\xymatrix{ X^k \ar[r]^{(\io_x)^k} \ar[d]^\pi & A^k \ar[d]^{\mu} \\ X_k \ar[r]^{(\io_x)_k} & A}$$
where $\pi$ is the natural projection, $\mu$ is the addition and $(\io_x)_k$ is the induced morphism. 
For $1 \leq k \leq g$, 
$$W_k(x) := (\io_x)_k(X_k)$$
defines an algebraic $k$-cycle on $A$. 
As before, we have 
$W_k(x)-W_k^\pm(y) \ \in \CH_k(A)_\hom$
for any $x,y \in X$.

\begin{remark}
Since $W_k(x)-W_k(y) \in \CH_k(A)_\alg$, 
the class of $W_k(x)-W_k^-(y)$ modulo algebraic equivalence does
not depend on $x$ and $y$. 
\end{remark}

Since $\pi$ is a finite morphism of degree $k!$, we have 
\begin{equation}\label{W_k}
k! \cdot W_k(x) = W_1(x)^{*k}
\end{equation}
where ${}^{*k}$ denotes the $k$-fold Pontryagin product. 
Since the Pontryagin product on the Chow ring is commutative and compatible with the inversion, 
we have: 

\begin{lemma}\label{factorization} 
We have identities in $\CH_k(A)$ 
$$k! (W_k(x) - W_k^\pm(y)) = (W_1(x)-W_1^\pm(y)) * V^\pm_{k-1}(x,y),$$
where we put 
$$V_{k-1}^\pm(x,y) = \sum_{0 \leq i \leq k-1} W_1(x)^{*i}  * W_1^\pm(y)^{*(k-1-i)}. $$
\end{lemma}

We put 
$$w_1= \cl_1(W_1^\pm(x)) \in H_2(A,\Z)$$
which is independent of $x$.  
The homology classes of $V_{k-1}(x,y)$ and $V_{k-1}^-(x,y)$
are independent of $x$, $y$, and agree with  
$$v_{k-1}:=  k \cdot w_1^{*(k-1)} \in H_{2(k-1)}(A,\Z).$$

\begin{lemma}
Under the identification 
$H_{2(k-1)}(A,\Z) \simeq \Hom(\wedge^{2(k-1)} H_\Z^1,\Z)$, 
for $\vphi=\vphi_1 \wedge \cdots \wedge \vphi_{2(k-1)} \in \wedge^{2(k-1)}H^1_\Z$, we have 
$$(v_{k-1}, \vphi) = k!  \sum_{\s} \sgn(\s) \prod_{i=1}^{k-1} \inner{\vphi_{\s(2i-1)}}{\vphi_{\s(2i)}}$$
where $\s$ runs through the elements of $S_{2(k-1)}$ such that 
$\s(2i-1)<\s(2i)$ for  $1 \leq i \leq k-1$ and $\s(2i-1) < \s(2i+1)$ for $1 \leq i \leq k-2$. 
\end{lemma}
\begin{proof}
Use Lemma \ref{dual-pont} recursively with $p=2$. 
\end{proof}

The Pontryagin product 
$$* \colon H_3(A)(-1) \ot H_{2(k-1)}(A)(-k+1) \ra H_{2k+1}(A)(-k)$$ 
is a morphism of Hodge structures of weight $-1$ and induces a homomorphism
\begin{equation}\label{pont-int-jac}
* \colon J_1(A) \ot \Hom(\Z, H_{2(k-1)}(A)(-k+1)) \ra J_k(A). 
\end{equation}

\begin{lemma} We have an identity in $J_k(A)$ 
$$k! \cdot \Phi_k(W_k(x)-W_k^\pm(y))=\Phi_1(W_1(x)-W_1^\pm(y)) * v_{k-1} .$$
\end{lemma}

\begin{proof}
By \eqref{W_k} and Lemma \ref{factorization}, 
it suffices to show the commutativity of the diagram
$$\xymatrix{
\CH_1(A)_\hom \ot \CH_{k-1}(A) \ar[r]^(.67){*} \ar[d]_{\Phi_1 \ot \cl_{k-1}} & \CH_{k}(A)_\hom \ar[d]_{\Phi_k}\\
J_1(A) \ot \Hom(\Z, H_{2(k-1)}(A)(-k+1))  \ar[r]^(.67){*}   &J_k(A)
}$$
which is clear since $\pd W \times Z = \pd(W \times Z)$ for a topological $3$-chain $W$ and a topological $2(k-1)$-cycle $Z$,  
and $\Phi_k$ commutes with $\mu_*$. 
\end{proof}

By \eqref{int-jac-real} and \eqref{identification_homology}, 
we have
\begin{equation}\label{int-jac-hom}
J_k(A) \simeq \Hom(\wedge^{2k+1}H^1_\Z,\R/\Z),
\end{equation}
so \eqref{pont-int-jac} is identified with a map
\begin{equation}\label{ext-product}
\Hom(\wedge^3H^1_\Z,\R/\Z) \ot \Hom((\wedge^{2(k-1)} H^1)(k-1), \Z) \ra \Hom(\wedge^{2k+1}H^1_\Z, \R/\Z). 
\end{equation}
This map is the restriction of the map induced from the homomorphism 
$$\pi_{3,2(k-1)}\colon \wedge^{2k+1}H^1_\Z \ra \wedge^3 H^1_\Z \ot \wedge^{2(k-1)} H^1_\Z$$
of Lemma \ref{dual-pont}. Therefore: 
\begin{lemma}  
The image of $f \ot g$ under \eqref{ext-product} sends 
$\vphi_1 \wedge \cdots \wedge \vphi_{2k+1} \in \wedge^{2k+1} H^1_\Z$ to 
$$\sum_{\s \in S_{3,2(k-1)}} \sgn(\s) f(\vphi_{\s(1)} \wedge \vphi_{\s(2)} \wedge \vphi_{\s(3)})
g(\vphi_{\s(4)} \wedge \cdots \wedge \vphi_{\s(2k+1)}).$$
\end{lemma}

From the above lemmas, we obtain: 
\begin{proposition}\label{ajimage-k} 
Under the identification \eqref{int-jac-hom}, we have 
for $\vphi_i \in H^1_\Z \ (1 \leq i \leq 2k+1)$, 
\begin{multline*}
k! \cdot \Phi_k(W_k(x)-W_k^\pm(y))(\vphi_1 \wedge \cdots \wedge \vphi_{2k+1})
\\=
k!  \sum_{\s} \sgn(\s)
\Phi_1(W_1(x)-W_1^\pm(y)) ( \vphi_{\s(1)} \wedge  \vphi_{\s(2)} \wedge  \vphi_{\s(3)} ) 
\prod_{i=1}^{k-1} \inner{\vphi_{\s(2i+2)}}{\vphi_{\s(2i+3)}} 
\end{multline*}
where $\s$ runs through the elements of $S_{2k+1}$ such that
$\s(1)<\s(2)<\s(3)$, $\s(2i+2) < \s(2i+3)$ for  $1 \leq i \leq k-1$, and 
$\s(2i+2) < \s(2i+4)$ for $1 \leq i \leq k-2$. 
\end{proposition}

%%%%%%%%%%%%%%%%%%%%%%%%%%%%%%%%%%%%%%%%%%%%%%
%%%%%%%%%%%%%%%%%%%%%%%%%%%%%%%%%%%%%%%%%%%%%%
\section{Fermat curve}

In this section, after recalling basic materials,  
we calculate iterated integrals and the pointed harmonic volume of the Fermat curve.  
Then, we derive a description of the Abel-Jacobi image of the Ceresa cycle. 
Finally, we introduce examples of cohomology classes of Hodge type $(k+2,k-1)+(k-1,k+2)$.   

%%%%%%%%%%%%%%%%%%%%%
\subsection{Iterated integral on the Fermat curve}

Let $N\ge 4$ be an integer and let 
$$X \colon x_0^N + y_0^N = z_0^N$$
be the Fermat curve of degree $N$ and $x^N+y^N=1$ be its affine equation. 
As a base point, we choose $(x,y)=(0,1)$. 
Its genus is $g=(N-1)(N-2)/2 \ge 3$ by our assumption. 
Let $\mu_N$ be the group of $N$-th roots of unity in $\C$ and put $\z=\exp(2\pi i/N)$. 
Put a group 
$$G=\mu_N \times \mu_N$$ 
and its elements $\a=(\z,1)$, $\b=(1,\z)$. 
Let $G$ act on $X$ by 
$\a^r\b^s(x,y) = (\z^rx,\z^sy)$. 
If $R$ is a ring, then the group ring $R[G]$ acts on $H_1(X,R)$ (resp. $H^1(X,R)$) by the push-forward (resp. pull-back). 

Define a path by  
$$\d \colon [0,1] \ra X; \quad t \mapsto (t^{1/N}, (1-t)^{1/N})$$
where the branches are taken in $\R_{\geq 0}$. 
If we put
$$\k = \d \cdot (\b_*\d)^{-1} \cdot (\a\b)_*\d \cdot (\a_*\d)^{-1},$$ 
it is a loop with the base point $(0,1)$. 
Then, $H_1(X,\Z)$ is a cyclic $\Z[G]$-module generated by $\k$ (Rohrlich \cite{rohrlich}). 

We put an index set
$$\bI=\{(a,b) \in (\Z/N\Z)^{\oplus 2} \mid a,b,a+b \neq 0\}. $$
For $a \in \Z/N\Z-\{0\}$, let $\angle a \in \{1,2, \dots, N-1\}$ denote its representative. 
For $(a,b)\in \bI$, we define a differential form by 
$$\o_0^{a,b} = x^{\angle a}y^{\angle{b}-N} \dlog x.$$ 
Note that $\o_0^{a,b}$ is an eigenvector for the $G$-action: 
$$(\a^r\b^s)^* \o_0^{a,b} = \z^{ar+bs} \o_0^{a,b}.$$ 
It is a differential form of the second kind (i.e. has no residues) and defines a cohomology class, 
which by abuse of notation we denote by the same letter. It is of the first kind (i.e. holomorphic) 
if and only if $(a,b)$ belongs to
$$\bI_\holo :=\{(a,b) \in \bI \mid \angle{a}+\angle{b}<N\}.$$
Note that for each $(a,b) \in \bI$, exactly one of $(a,b)$, $(-a,-b)$ is in $\bI_\holo$.  
It is well-known that 
$$\{\o_0^{a,b} \mid (a,b) \in \bI\}, \quad \{\o_0^{a,b} \mid (a,b) \in \bI_\holo\}, \quad 
\{\o_0^{-a,-b} \mid (a,b) \in \bI_\holo\}$$ 
are bases of $H^1(X,\C)$, $H^{1,0}(X)$, $H^{0,1}(X)$, respectively. 

We have 
$$\int_\d \o_0^{a,b} = \tfrac{1}{N}B\bigl(\tfrac{\angle{a}}{N}, \tfrac{\angle{b}}{N}\bigr)$$
where $B(s,t)$ is the Beta function. If we normalize as 
$$\o^{a,b} = \Bigl(\tfrac{1}{N}B\bigl(\tfrac{\angle{a}}{N}, \tfrac{\angle{b}}{N}\bigr)\Bigr)^{-1} \o_0^{a,b},$$
then we have $\int_{(\a^r\b^s)_* \d}\o^{a,b}=\z^{ar+bs}$, hence the periods
\begin{equation}\label{period}
\quad \int_{(\a^r\b^s)_*\k}\o^{a,b} = \z^{ar+bs} (1-\z^{a})(1-\z^{b}) . 
\end{equation}

By using \eqref{formula-1}, \eqref{formula-2}, we have (\cite{tadokoro-2}, Lemma 3.5)
\begin{equation*}
\int_\k \o^{a,b}\o^{c,d} = (1-\z^{a+c})(1-\z^{b+d})\int_\d \o^{a,b}\o^{c,d} + (1-\z^b)(\z^{a+c}+\z^{c+d}-\z^c-\z^d).
\end{equation*}
As a loop with the base point $(0,1)$ representing the homology class of $(\a^r\b^s)_*\k$, we choose 
$$\k^{r,s}:= \d \cdot ((\b^s)_*\d)^{-1} \cdot (\a^r\b^s)_*\k \cdot (\b^s)_*\d \cdot \d^{-1}.$$
Then we have similarly using \eqref{formula-3} 
\begin{multline*}
\int_{\k^{r,s}} \o^{a,b}\o^{c,d} = \z^{(a+c)r+(b+d)s} \int_\k \o^{a,b}\o^{c,d} \\
-\z^{ar+bs}(1-\z^a)(1-\z^b)(1-\z^{ds}) + \z^{cr+ds}(1-\z^c)(1-\z^d)(1-\z^{bs}) 
\end{multline*}
 (\cite{tadokoro-2}, Theorem 3.6).  
We have proved: 

\begin{proposition}\label{it-int} 
Let $(a,b), (c,d) \in \bI$ and $\k^{r,s}$ ($r,s \in \Z/N\Z$) be the loop as above. 
Then we have
\begin{multline*}
\int_{\k^{r,s}} \o^{a,b}\o^{c,d}
=  
\z^{(a+c)r+(b+d)s}\left\{(1-\z^{a+c})(1-\z^{b+d})\int_\d \o^{a,b}\o^{c,d} +P_1\right\} \\
+ \z^{ar+bs}(1-\z^{ds})P_2 + \z^{cr+ds}(1-\z^{bs})P_3,
\end{multline*}
where $P_i$ are polynomials in $\z^a$, $\z^b$, $\z^c$, $\z^d$ with $\Z$-coefficients independent of $r, s$. 
\end{proposition}

%%%%%%%%%%%%%%%%%%%%%%%%%%%%%%
\subsection{Harmonic volume of the Fermat curve}

Let $K=\Q(\mu_N)$ be the $N$-cyclotomic field, $\sO$ be its integer ring 
and fix a primitive $N$-th root of unity $\xi \in \sO$. 
We use notations such as $H_\sO = H \ot_\Z \sO$, $H_K = H \ot_\Z K$, $H_\C = H \ot_\Z \C$. 
Since we fixed the base point $(0,1)$, we denote the pointed harmonic volume just by $m$.  
It extends naturally to
$$m_\sO \colon K_\sO \ot H^1_\sO \ra (\sO \ot_\Z \R) / \sO. $$  
Note that 
$$ \sO \ot_\Z \R \simeq \left[\prod_{\s \colon K \hookrightarrow \C} \C \right]^+$$
where $\s$ runs through the embeddings of $K$ into $\C$ and $+$ denotes the fixed part by the complex conjugation acting on the set $\{\s\}$ and $\C$ at the same time. 
Let $m_\s$ denote the $\s$-component of $m_\sO$. 
The non-triviality of $m$ follows from that of $m_\s$ or the composition $\Tr \circ m_\sO = \sum_\s m_\s$,  where 
$$\Tr \colon (\sO \ot_\Z \R) / \sO \ra \R/\Z$$
is the trace map. 

For each $(a,b) \in \bI$, define a character by
$$\theta^{a,b} \colon G \ra \sO^*; \quad \theta^{a,b}(\a^r\b^s) = \xi^{ar+bs}$$
and let
$$p^{a,b} = \frac{1}{N^2} \sum_{g\in G} \theta^{a,b}(g^{-1})g \ \in K[G]$$
be the projector to the $\theta^{a,b}$-isotropic component. 
By the structure of $H_1(X,\Z)$ and $H^1(X,\C)$ mentioned above and the Poincar\'e duality,  
$H_1(X,k)$ admits a basis $\{(p^{a,b})_* \k  \mid (a,b) \in \bI\}$. 

For an embedding $\s \colon K \hookrightarrow \C$, let $h \in (\Z/N\Z)^*$ be the element such that $\s(\xi)=\z^h$. By \eqref{period}, $\o^{a,b}$ is in the image of  $H^1_\sO \subset H^1_k \us{\s}{\hookrightarrow} H^1_\C$. 
Let $\vphi^{a,b} \in H^1_K$ be the element corresponding under the Poincar\'e duality to 
$$\frac{N^2}{1-\xi^{-(a+b)}} (p^{-a,-b})_*\k \in H_{1,K}. $$ 

\begin{proposition}\label{varphi-ab} \ 
\begin{enumerate} 
\item
For an embedding $\s \colon K \hookrightarrow \C$ with $\s(\xi)=\z^h$,  we have
$\s(\vphi^{a,b})=\o^{ha,hb}$, hence $\vphi^{a,b} \in H^1_\sO$. 
\item 
Under the pairing 
$\langle\ , \ \rangle_\sO \colon H^1_\sO \wedge H^1_\sO \ra \sO$, 
we have
$$\inner{\vphi^{a,b}}{\vphi^{c,d}}=\begin{cases}
N^2 \frac{(1-\xi^{a})(1-\xi^{b})}{1-\xi^{a+b}} & \text{if \ $(a,b)=(-c,-d)$},\\
0 &  \text{otherwise}.
\end{cases}
$$
\end{enumerate}
\end{proposition}

\begin{proof} 
Under the Poincar\'e duality, $g_*$ on the homology corresponds to $(g^{-1})^*$ on the cohomology. 
Therefore, $(\s(p^{-a,-b}))_*\k$ correspond to a constant multiple of $\o^{ha,hb}$. 
By \eqref{period}, we have 
$$\int_{\s(p^{-a,-b})_*\k} \o^{c,d} = \int_\k \s(p^{-a,-b})^* \o^{c,d} = (1-\z^{-ha})(1-\z^{-hb})$$ 
if $(c,d)=(-ha,-hb)$, and $0$ otherwise. 
On the other hand, one calculates 
\begin{equation*}
\langle \o^{c,d}, \o^{ha,hb} \rangle 
= N^2 \frac{(1-\z^{-ha})(1-\z^{-hb})}{1-\z^{-h(a+b)}}
\end{equation*}
if $(c,d)=(-ha,-hb)$, and $0$ otherwise. 
By comparison, we obtain the result. 
\end{proof}

\begin{remark}\label{denominator}
It follows from \eqref{period} that, for any $(a, b) \in \bI$, we have
$$\frac{1}{(1-\xi^{a})(1-\xi^{b})} \ \vphi^{a,b} \in H_1(X,\sO).$$ 
\end{remark}

We restrict ourselves to the following situation. 
\begin{assumption}\label{assumption} 
Let $(a_i,b_i) \in \bI$ $(i=1,2,3)$ and assume: 
\begin{enumerate} 
\item $\sum_{i=1}^3 (a_i,b_i)=(0,0)$.   
\item For any $h \in (\Z/N\Z)^*$, we have either $(ha_i,hb_i) \in \bI_\holo$ for $i=1,2$, or 
$(ha_i,hb_i) \not\in \bI_\holo$ for $i=1,2$.
\end{enumerate} 
\end{assumption}

The assumption (i) implies that $(a_1,b_1)\neq (-a_2,-b_2)$, hence 
$$\vphi^{a_1,b_1}\ot \vphi^{a_2,b_2} \in K_\sO$$ 
by Proposition \ref{varphi-ab} (ii). 
The assumption (ii) is satisfied for example if $(a_1,b_1)=(b_2,a_2)$. 

\begin{proposition}\label{harm-vol-sigma}
Under Assumption \ref{assumption}, let $\s \colon K \hookrightarrow \C$ be an embedding with $\s(\xi)=\z^h$ 
such that $(ha_i,hb_i)\in \bI_\holo$ for $i=1,2$. 
Then we have
\begin{equation*}
m_\s(\vphi^{a_1,b_1}\ot \vphi^{a_2,b_2} \ot \vphi^{a_3,b_3}) 
= \frac{N^2(1-\z^{-ha_3})(1-\z^{-hb_3})}{1-\z^{-h(a_3+b_3)}} \int_\d \o^{ha_1,hb_1} \o^{ha_2,hb_2} + P(\z^h)
\end{equation*}
where $P$ is a polynomial with $\Z$-coefficients independent of $\s$.  
\end{proposition}

\begin{proof}
By Assumption \ref{assumption} (ii) and the assumption on $\s$, we have 
$\eta=0$ in Theorem \ref{pulte}. 
By Proposition \ref{varphi-ab} (i), we are reduced to calculate 
$$ \frac{1}{1-\z^{-h(a_3+b_3)}} \sum_{r,s \in \Z/N\Z} \z^{h(a_3r+b_3s)}\int_{\k^{r,s}} \o^{ha_1,hb_1}\o^{ha_2,hb_2}.$$
Then, by Proposition \ref{it-int} and Assumption \ref{assumption} (i), we obtain the result. 
\end{proof}

\begin{remark}
If Assumption \ref{assumption} (i)  is not satisfied, still assuming $(a_1,b_1)\neq (-a_2,-b_2)$, 
then the proof of Proposition \ref{harm-vol-sigma} shows that 
$m_\sO(\vphi^{a_1,b_1}\ot \vphi^{a_2,b_2} \ot \vphi^{a_3, b_3})=0$. 
\end{remark}

\begin{corollary}\label{criterion}
If  
$$\int_\d \o^{ha_1,hb_1} \o^{ha_2,hb_2}\not\in \Q(\mu_N)$$ 
for some $h \in (\Z/N\Z)^*$, then, $W_k-W_k^-$ is non-torsion modulo algebraic equivalence for $1 \leq k \leq g-2$. 
\end{corollary}

\begin{proof}
It is immediate if $h$ satisfies the assumption of the proposition. 
The harmonic volume has the following property \cite{harris-book}:  
if $\inner{\vphi_i}{\vphi_j}=0$ for any $1\leq i,j \leq 3$, then for $\tau \in S_3$, we have
$$m(x)(\vphi_{\tau(1)} \ot \vphi_{\tau(2)} \ot \vphi_{\tau(3)}) = \sgn(\s) m(x)(\vphi_1\ot\vphi_2\ot\vphi_3).$$
Therefore we can assume that $(ha_1,hb_1)$ and $(ha_2,hb_2)$ have the same type. 
Then, since $m_{\ol\s} = \ol{m_\s}$, we are reduced to the first case. 
\end{proof}

\begin{theorem}\label{trace-harmonic}
Under Assumption \ref{assumption}, we have
$$\Tr \circ m_\sO \left(\frac{\vphi^{a_1,b_1}\ot \vphi^{a_2,b_2} \ot \vphi^{a_3, b_3}}{(1-\xi^{-a_3})(1-\xi^{-b_3})}\right) 
= N^2 \sum \int_\d \o^{ha_1,hb_1} \o^{ha_2,hb_2}$$
where the sum is taken over $h \in (\Z/N\Z)^*$ such that $(ha_i,hb_i) \in \bI_\holo$ for $i=1,2$. 
\end{theorem}

\begin{proof}
Note that, by the assumption and Remark \ref{denominator}, the element in the parentheses is an element of $K_\sO\ot H^1_\sO$.  
Since $m_{\ol\s} = \ol{m_\s}$ and the iterated integral in Proposition \ref{harm-vol-sigma} is $\R$-valued, we are reduced to 
$\frac{1}{1-\z} + \frac{1}{1-\z^{-1}} =1$ for a root of unity $\z$ other than $1$.  
\end{proof}

Now, let $1 \leq k \leq g-2$ and $(a_i,b_i) \in \bI \ (1 \leq i \leq 2k+1)$ be distinct elements such that 
$\{(a_i,b_i) \mid i=1,2,3\}$ satisfies Assumption \ref{assumption} and $(a_{2i+2}, b_{2i+2})=(-a_{2i+3},-b_{2i+3})$ for $1 \leq i \leq k-1$. 
Put 
$$\vphi_1=\vphi^{a_1,b_1}, \quad \vphi_2=\vphi^{a_2,b_2}, \quad \vphi_3=\frac{1}{(1-\xi^{-a_3})(1-\xi^{-b_3})} \ \vphi^{a_3,b_3},$$
and, for $1 \leq i \leq k-1$, 
$$\vphi_{2i+2} = \frac{1-\xi^{a+b}}{(1-\xi^a)(1-\xi^b)}\ \vphi^{a_{2i+2},b_{2i+2}}, \quad
\vphi_{2i+3} = \vphi^{a_{2i+3},b_{2i+3}},$$
so that $\inner{\vphi_{2i+2}}{\vphi_{2i+3}}=N^2$ by Proposition \ref{varphi-ab} (ii). 
Put $\vphi = \vphi^{a_1,b_1} \wedge \cdots \wedge \vphi^{a_{2k+1},b_{2k+1}}$. 

\begin{corollary}\label{trace-harmonic-k} 
With the notations as above, we have
$$k! \cdot \Tr \circ \Phi_k(W_k-W_k^-)(\vphi) = k! \cdot 2N^{2k}\sum \int_\d \o^{ha_1,hb_1} \o^{ha_2,hb_2}$$
where the sum is as in Theorem \ref{trace-harmonic}. 
\end{corollary}

\begin{proof}
This follows from Theorem \ref{harris-pulte}, Proposition \ref{ajimage-k} and Theorem \ref{trace-harmonic}.  
\end{proof}

%%%%%%%%%%%%%%%%%%%%%%%%%%%%%
\subsection{Examples}

For the non-triviality of $W_k-W_k^-$ modulo algebraic equivalence, not only rational equivalence, 
we consider the following case.   
\begin{proposition}\label{F^{k+2}}
Let $\vphi \in \wedge^{2k+1} H^1_\sO$ be as in Corollary \ref{trace-harmonic-k}. 
Assume further that for any $h \in (\Z/N\Z)^*$, we have either $(ha_i,hb_i) \in \bI_\holo$ for $i=1,2,3$, or 
$(ha_i,hb_i) \not\in \bI_\holo$ for $i=1,2,3$. 
Then we have 
$$\vphi \in \left(\wedge^{2k+1} H^1_\Z \cap (F^{k+2}+ \ol{F^{k+2}})\right) \ot_\Z \sO.$$ 
\end{proposition}
\begin{proof}
It follows from Proposition \ref{varphi-ab} (i) and our assumptions.
\end{proof}

Therefore, the non-vanishing of $\Phi_k(W_k-W_k^-)(\vphi)$ for such $\vphi$ implies that $W_k-W_k^- \not\in \CH_k(A)_\alg$. 
Now we give such examples. Note that it suffices to give $\{(a_i,b_i) \mid i=1,2,3\}$,  
since we can always find the others by the assumption $k \leq g-2$. 

\begin{example}\label{example}
If we put 
$$(a_1, b_1)=(1,-2),\quad (a_2,b_2)=(-2,1), \quad (a_3,b_3)=(1,1),$$
then they satisfy the assumption of Proposition \ref{F^{k+2}}. We have 
$(ha_i,hb_i) \in \bI_\holo$ for $i=1,2,3$ if and only if $\angle{h}<N/2$. 
\end{example}

\begin{example}\label{klein}
Let $N=7$ and put
$$(a_1,b_1)=(1,2), \quad (a_2,b_2)=(2,4), \quad (a_3,b_3)=(4,1).$$ 
Then they satisfy the assumption of Proposition \ref{F^{k+2}}.  
We have $(ha_i,hb_i) \in \bI_\holo$ for $i=1,2,3$ if and only if $h=1,2,4$. 

It is related the study of the Klein quartic curve 
$$C \colon u_0^3v_0+v_0^3w_0+w_0^3u_0=0$$
whose genus is $3$. 
Consider the unramified morphism of degree $7$ 
$$f \colon X \ra C; \quad (u_0:v_0:w_0) = (x_0^3z_0:-y_0^3x_0:z_0^3y_0).$$ 
There is a unique holomorphic $1$-form $\o_i$ on $C$ such that $f^*\o_i = \o^{a_i,b_i}$  
and $H_1(C,\Z)$ is generated by $\{f_*((\a^r)_* (\d \cdot (\a_*\d)^{-1})) \mid r \in \Z/7\Z\}$ 
(cf. \cite{tadokoro}). 
Since 
$$\int_{f_*((\a^r)_* (\d \cdot (\a_*\d)^{-1}))} \o_i =
\int_{(\a^r)_* (\d \cdot (\a_*\d)^{-1})} \o^{a_i,b_i} =
 \z^{ra_i}(1-\z^{a_i}), $$
 we have
$$\frac{\vphi^{1,2}\wedge \vphi^{2,4} \wedge \vphi^{4,1}}{(1-\xi^{-4})(1-\xi^{-1})} \in f^*(\wedge^3H^1(C,\sO)).$$
Therefore, if the quantity of Corollary \ref{trace-harmonic-k} is not a rational integer (which will be verified in \S5.3), 
then the Ceresa cycle ($k=1$) of $C$ is non-trivial modulo algebraic equivalence, 
which reproves Tadokoro's theorem \cite{tadokoro}. 
\end{example}

\begin{remark}
In fact, it is known that the Ceresa cycle ($k=1$) of $C$ is non-torsion modulo algebraic equivalence. 
In \cite{kimura}, Kimura proves that the image of the modified diagonal cycle 
of Gross-Schoen \cite{gross-schoen} under the $\ell$-adic Abel-Jacobi map (for some prime $\ell$) is non-torsion, and the image equals the image of the Ceresa cycle times $3$ (Colombo-van Geemen \cite{c-vg}, Proposition 2.9). 
\end{remark}

%%%%%%%%%%%%%%%%%%%%%%%%%%%%%%%%%%%%%%%%%%%%%%%%%
%%%%%%%%%%%%%%%%%%%%%%%%%%%%%%%%%%%%%%%%%%%%%%%%%
\section{Generalized hypergeometric functions}

In this section, after recalling generalized hypergeometric functions of Barnes and Appell, 
we express the harmonic volume of the Fermat curve 
in terms of special values of those functions. 
Finally, we compute examples and prove Theorem \ref{main-theorem}. 

%%%%%%%%%%%%%%%%%%%%%%%%%%%%%%%
\subsection{Barnes' hypergeometric function ${}_3F_2$}

Let 
$$(a,n) = a(a+1)\cdots (a+n-1) = \vG(a+n)/\vG(a)$$ 
be the Pochhammer symbol, where $\vG(s)$ is the Gamma function. For $d, e \not\in \{0,-1, -2, \dots\}$, 
Barnes hypergeometric function ${}_3F_2$ is defined by 
\begin{equation*}
\Fx{a}{b}{c}{d}{e} = \sum_{n\ge 0} \frac{(a,n)(b,n)(c,n)}{(d,n)(e,n)(1,n)} x^n .
\end{equation*} 
Its radius of convergence is $1$, but if $\mathrm{Re}(d+e-a-b-c)>0$, it also converges for 
$|x|=1$. 
We use the standard notation
$$\G{\a_1, \dots, \a_m}{\b_1,\dots, \b_n} = \frac{\vG(\a_1)\cdots \vG(\a_m)}{\vG(\b_1)\cdots \vG(\b_n)}.$$
By the integral representation of Euler type (cf. \cite{slater}), we have  
\begin{multline*}
\int_{[0,1]} u^{\a_1-1}(1-u)^{\b_1-1}du \,  v^{\a_2-1}(1-v)^{\b_2-1} dv
\\ 
= \G{\a_1, \b_2, \a_1+\a_2}{\a_1+1,\a_1+\a_2+\b_2} 
\F{\a_1}{1-\b_1}{\a_1+\a_2}{\a_1+1}{\a_1+\a_2+\b_2}.
\end{multline*}
By using Dixon's formula (cf. \cite{slater}) repeatedly, we obtain ten such expressions. 
\begin{lemma}\label{dixon}
If $0 < \a_i, \b_i < 1$, then we have
\allowdisplaybreaks{
\begin{align*}
&\int_{[0,1]} u^{\a_1-1}(1-u)^{\b_1-1}du \,  v^{\a_2-1}(1-v)^{\b_2-1}dv\\ 
&= \G{\a_1, \b_2, \a_1+\a_2}{\a_1+1,\a_1+\a_2+\b_2} 
\F{\a_1}{1-\b_1}{\a_1+\a_2}{\a_1+1}{\a_1+\a_2+\b_2}\\ 
&= \G{\a_1,\b_2,\b_1+\b_2}{\b_2+1,\a_1+\b_1+\b_2}
\F{1-\a_2}{\b_2}{\b_1+\b_2}{\b_2+1}{\a_1+\b_1+\b_2}\\
&=\G{\a_1,\b_2,\a_1+\a_2,\b_1+\b_2}{\a_1+1,\a_2+\b_2,\a_1+\b_1+\b_2}
\F{\a_1}{1-\a_2}{\a_1+\b_1}{\a_1+1}{\a_1+\b_1+\b_2}\\
&= \G{\a_1,\b_2,\a_1+\a_2,\b_1+\b_2}{\b_2+1,\a_1+\b_1,\a_1+\a_2+\b_2}
\F{1-\b_1}{\b_2}{\a_2+\b_2}{\b_2+1}{\a_1+\a_2+\b_2}\\
&= \G{\a_1,\a_1+\a_2,\b_1+\b_2}{\a_1+1,\a_1+\a_2+\b_1+\b_2} 
\F{\a_1+\a_2}{\a_1+\b_1}{1}{\a_1+1}{\a_1+\a_2+\b_1+\b_2}\\
&= \G{\a_1+\a_2,\b_2, \b_1+\b_2}{\b_2+1, \a_1+\a_2+\b_1+\b_2} 
\F{\a_2+\b_2}{\b_1+\b_2}{1}{\b_2+1}{\a_1+\a_2+\b_1+\b_2}\\
&= \G{\a_1, \b_2, \a_1+\a_2, \b_1+\b_2}{1-\a_2, \a_1+\a_2+\b_2, \a_1+\a_2+\b_1+\b_2} \\* 
& \phantom{AAAAAAAAAAAAA}
\times \F{\a_1+\a_2}{\a_2+\b_2}{\a_1+\a_2+\b_1+\b_2-1}{\a_1+\a_2+\b_2}{\a_1+\a_2+\b_1+\b_2}\\
&= \G{\a_1,\b_2,\a_1+\a_2,\b_1+\b_2}{1-\b_1, \a_1+\b_1+\b_2, \a_1+\a_2+\b_1+\b_2} \\* 
& \phantom{AAAAAAAAAAAAA}
\times \F{\a_1+\b_1}{\b_1+\b_2}{\a_1+\a_2+\b_1+\b_2-1}{\a_1+\b_1+\b_2}{\a_1+\a_2+\b_1+\b_2}\\
&=\G{\a_1,\b_2,\a_1+\a_2, \b_1+\b_2}{\a_1+1,\b_2+1,\a_1+\a_2+\b_1+\b_2-1}
\F{1-\a_2}{1-\b_1}{1}{\a_1+1}{\b_2+1}\\
&= \G{\a_1, \b_2, \a_1+\a_2,\b_1+\b_2}{\a_1+\a_2+\b_2,\a_1+\b_1+\b_2} 
\F{\a_1}{\b_2}{\a_1+\a_2+\b_1+\b_2-1}{\a_1+\a_2+\b_2}{\a_1+\b_1+\b_2} 
\end{align*}
}
where the ninth one converges when $\a_1+\a_2+\b_1+\b_2<1$. 
\end{lemma}

We shall use the last one, which is most symmetric and has the rapidest convergence. 

%%%%%%%%%%%%%%%%%%%%%%%%%%%%
\subsection{Appell's hypergeometric function $F_3$}

Appell's hypergeometric function $F_3$ of two variables \cite{appell} is defined by
$$F_3(\a,\a',\b,\b',\g;x,y) = \sum_{m,n \ge 0} \frac{(\a,m)(\a',n)(\b,m)(\b',n)}{(\g,m+n)(1,m)(1,n)}\ x^my^n.$$
It converges at $(x,y)=(1,1)$ If $\mathrm{Re}(\g-\a-\b)>0$ and $\mathrm{Re}(\g-\a'-\b')>0$.   
When $\g=\a+\a'+1$, by comparing the integral representations of ${}_3F_2$ and $F_3$, we have 
\begin{multline*}
F_3(\a,\a',\b,\b',\a+\a'+1;x,1)\\
= \G{\a+\a'+1, \a-\b'+1}{\a+1,\a+\a'-\b'+1} \Fx{\a}{\b}{\a-\b'+1}{\a+1}{\a+\a'-\b'+1}.
\end{multline*}
In particular, we have
\begin{multline*}
\int_{[0,1]} u^{\a_1-1}(1-u)^{\b_1-1}du \,  v^{\a_2-1}(1-v)^{\b_2-1} dv
\\ = \G{\a_1,\b_2}{\a_1+\b_2+1} F_3(\a_1,\b_2,1-\b_1,1-\a_2,\a_1+\b_2+1; 1,1). 
\end{multline*}

\begin{remark}
Similar special values of ${}_3F_2$ and $F_3$ also appear in the description of the Beilinson regulator of motives associated Fermat curves \cite{otsubo}. 
\end{remark}

%%%%%%%%%%%%%%%%%%%%%%%%%%%%%%%
\subsection{Harmonic volume and hypergeometric values}

Let the notations be as in \S 4. 
Since 
$$\d^* \o_0^{a,b} =  t^{\frac{\angle{a}}{N}}(1-t)^{\frac{\angle{b}}{N}-1} \frac{1}{N} \frac{dt}{t}$$
and by the well-known formula 
$$B(s,t) = \frac{\vG(s)\vG(t)}{\vG(s+t)},$$ 
we obtain from the last expression of Lemma \ref{dixon}: 
\begin{proposition}\label{it-int-fermat} 
For $(a_i,b_i) \in \bI$, put $\a_i =\angle{a_i}/N$, $\b_i =\angle{b_i}/N$. 
Then we have
\begin{multline*}
\int_\d \o^{a_1,b_1} \o^{a_2,b_2} 
= \G{\a_1+\a_2, \b_1+\b_2, \a_1+\b_1, \a_2+\b_2}{\a_2,\b_1,\a_1+\a_2+\b_2, \a_1+\b_1+\b_2} \\
\times \F{\a_1}{\b_2}{\a_1+\a_2+\b_1+\b_2-1}{\a_1+\a_2+\b_2}{\a_1+\b_1+\b_2} .
\end{multline*}
\end{proposition}

Now we apply this to Example \ref{example}. 
For an integer $h$ with $0<h<N/2$, $(h,N)=1$, Proposition \ref{it-int-fermat} reads 
\begin{equation*}
\int_\d \o^{h,-2h} \o^{-2h,h} 
= \frac{\vG\bigl(1-\tfrac{h}{N}\bigr)^4}{\vG\bigl(1-\tfrac{2h}{N}\bigr)^2} \  
\F{\tfrac{h}{N}}{\tfrac{h}{N}}{1-\tfrac{2h}{N}}{1}{1}
\end{equation*}
and by Corollary \ref{criterion}, Theorem \ref{main-theorem} (i) follows. 
Corollary \ref{trace-harmonic-k} reads 
\begin{equation}\label{value}
k!\cdot \Phi_k(W_k-W_k^-)(\vphi)= 
k! \cdot 2N^{2k} \sum_{0<h<N/2 \atop   (h,N)=1} \frac{\vG\bigl(1-\tfrac{h}{N}\bigr)^4}{\vG\bigl(1-\tfrac{2h}{N}\bigr)^2} \ 
\F{\tfrac{h}{N}}{\tfrac{h}{N}}{1-\tfrac{2h}{N}}{1}{1} . 
\end{equation}

Let $f(N,k)$ denote the right-hand side of \eqref{value} and we compute its fractional part. 
Recall that $N\geq 4$ and $1 \leq k \leq g-2$ where $g=(N-1)(N-2)/2$. 
Here is the command for Mathematica: 
\medskip
\begin{quote}
\begin{verbatim}
f[n_,k_] := k!2n^(2k)Sum[If[GCD[n, h]==1,1,0]
      Gamma[1-h/n]^4/Gamma[1-2h/n]^2 
          HypergeometricPFQ[{h/n,h/n,1-2h/n},{1,1},1], 
               {h,1,IntegerPart[(n-1)/2]}]
N[FractionalPart[f[n,k]]]
\end{verbatim}
\end{quote}
\medskip
The author verified that $f(N,1)\not\in\Z$ for $N \leq 1000$ (see the table below for $N<100$) 
and that $f(N,k)\not\in\Z$ for $N \leq 8$ and any $k$, hence Theorem \ref{main-theorem} (ii) is proved.  

In view of Swinnerton-Dyer's conjecture (see Introduction), 
one may also fix $N$, $k$ and compute $m f(N,k)$ for various $m$ to obtain a lower bound of the order of $W_k-W_k^-$ modulo algebraic equivalence.  
For example, the author verified that $m f(5,1)\not\in \Z$  for $m \leq 10^5$. 

\begin{table}[h]
\caption{The fractional part of $f(N,1)$ for $N<100$}
\begin{center}
\begin{tabular}{r|l}\label{table}
$N$ & $f(N,1)$ \\
\hline
$4$&    $0.262996$ \\
$5$&    $0.537741$ \\
$6$&    $0.834938$ \\
$7$&    $0.0389723$ \\
$8$&    $0.486831$ \\
$9$&    $0.191617$ \\
$10$&   $0.0194112$ \\
$11$&   $0.714331$ \\
$12$&   $0.787413 $\\
$13$&   $0.339364 $\\
$14$&   $0.107307 $\\
$15$&   $0.964777 $\\
$16$&   $0.0707329$ \\
$17$&   $0.849791 $\\
$18$&   $0.8478 $\\
$19$&   $0.216837$\\
$20$&   $0.459979$\\
$21$&   $0.296951$\\
$22$&   $0.876098$\\
$23$&   $0.884882$\\
$24$&   $0.565879$\\
$25$&   $0.227588$\\
$26$&   $0.674037$\\
$27$&   $0.024742$\\
$28$&   $0.860369$\\
$29$&   $0.862392$\\
$30$&   $0.706843$\\
$31$&   $0.753471$\\
$32$&   $0.389462$\\
$33$&   $0.736648$\\
$34$&   $0.106166$\\
$35$&   $0.518381$
\end{tabular}
\qquad
\begin{tabular}{r|l}
$N$ & $f(N,1)$ \\
\hline
$36$&   $0.447655$\\
$37$&   $0.525754$\\
$38$&   $0.709018$\\
$39$&   $0.90578$\\
$40$&   $0.885897$\\
$41$&   $0.888106$\\
$42$&   $0.664142$\\
$43$&   $0.053105$\\
$44$&   $0.194837$\\
$45$&   $0.167823$\\
$46$&   $0.581124$\\
$47$&   $0.0668079$\\
$48$&   $0.0527443$\\
$49$&   $0.492313$\\
$50$&   $0.316991$\\
$51$&   $0.298819$\\
$52$&   $0.59749$\\
$53$&   $0.444978$\\
$54$&   $0.919842$\\
$55$&   $0.714357$\\
$56$&   $0.197632$\\
$57$&   $0.321665$\\
$58$&   $0.688486$\\
$59$&   $0.0898551$\\
$60$&   $0.687806$\\
$61$&   $0.832525$\\
$62$&   $0.301712$\\
$63$&   $0.02593$\\
$64$&   $0.920061$\\
$65$&   $0.706527$\\
$66$&   $0.0810429$\\
$67$&   $0.0490554$
\end{tabular}
\qquad
\begin{tabular}{r|l}
$N$ & $f(N,1)$ \\
\hline
$68$&   $0.718085$\\
$69$&   $0.964278$\\
$70$&   $0.103166$\\
$71$&   $0.449617$\\
$72$&   $0.544859$\\
$73$&   $0.356497$\\
$74$&   $0.505994$\\
$75$&   $0.232621$\\
$76$&   $0.992762$\\
$77$&   $0.581805$\\
$78$&   $0.102977$\\
$79$&   $0.822496$\\
$80$&   $0.517871$\\
$81$&   $0.960151$\\
$82$&   $0.0135158$\\
$83$&   $0.686773$\\
$84$&   $0.791853$\\
$85$&   $0.862785$\\
$86$&   $0.698527$\\
$87$&   $0.169399$\\
$88$&   $0.440793$\\
$89$&   $0.678576$\\
$90$&   $0.312135$\\
$91$&   $0.285791$\\
$92$&   $0.877431$\\
$93$&   $0.360037$\\
$94$&   $0.796999$\\
$95$&   $0.797337$\\
$96$&   $0.532044$\\
$97$&   $0.848835$\\
$98$&   $0.898728$\\
$99$&   $0.72628$
\end{tabular}
\end{center}
\end{table}

Finally, if we apply Proposition \ref{it-int-fermat} to Example \ref{klein}, then 
Corollary \ref{trace-harmonic-k} reads 
\begin{multline*}
k! \cdot \Tr \circ \Phi_k(W_k-W_k^-)(\vphi)\\
= k!\cdot 2\cdot 7^{2k} 
\left(
\G{\frac{3}{7},\frac{6}{7}}{\frac{2}{7}}^2
+\G{\frac{5}{7},\frac{6}{7}}{\frac{4}{7}}^2
+\G{\frac{3}{7},\frac{5}{7}}{\frac{1}{7}}^2
\right)
\F{\tfrac{1}{7}}{\tfrac{2}{7}}{\tfrac{4}{7}}{1}{1}. 
\end{multline*}
For the maximal $k=13$, the fractional part of the right-hand side is 
$0.96275\pm 10^{-5}$. 
This shows again that $W_k-W_k^-$ is not algebraically trivial for $N=7$ and any $k$. 

%%%%%%%%%%%%%%%

\end{document}